\documentclass[reqno,a4paper,12pt]{amsart}

\usepackage{amssymb}

\numberwithin{equation}{section}
\date{\today}

\newcommand{\q}{\qquad}

\newcommand\x{\mathbf{x}}
\newcommand{\y}{\mathbf{y}}

\newcommand{\ft}{\mathfrak{t}}

\newcommand{\R}{\mathbb{R}}

\newcommand{\bbS}{\mathbb{S}}

\newcommand{\Con}{C_0^\infty}

\newcommand{\grad}{\nabla}

\newtheorem{Lemma}{Lemma}
\newtheorem{Theorem}{Theorem}
\newtheorem{Corollary}{Corollary}
\newtheorem{Proposition}{Proposition}

\newtheorem{Example}{Example}
\begin{document}

\title[Spectral Properties]{Spectral Properties of Some Degenerate
Elliptic Differential Operators}

\author[Roger T.~Lewis]{Roger T. Lewis}

\address{Department of Mathematics\\
         University of Alabama at Birmingham\\
         Birmingham, AL 35294-1170\\
         USA}
\email{lewis@math.uab.edu}

\keywords{Essential spectrum, Discrete spectrum, Hardy inequality,
Elliptic Operators, Distance function}

\subjclass{Primary 47F05; 47B25; Secondary 35P15, 35J70}

\dedicatory{This paper is dedicated to David Edmunds on his 80th
birthday and to\newline Desmond Evans on his 70th birthday.}

\begin{abstract}
{In this paper we extend classical criteria for determining lower
bounds for the least point of the essential spectrum of
second-order elliptic differential operators on domains
$\Omega\subset\R^n$ allowing for degeneracy of the coefficients on
the boundary. We assume that we are given a sesquilinear form and
investigate the degree of degeneracy of the coefficients near
$\partial\Omega$ that can be tolerated and still maintain a
closable sesquilinear form to which the First Representation
Theorem can be applied. Then, we establish criteria characterizing
the least point of the essential spectrum of the associated
differential operator in these degenerate cases. Applications are
given for convex and non-convex $\Omega$ using Hardy inequalities,
which recently have been proven in terms of the distance to the
boundary, showing the spectra to be purely discrete.}
\end{abstract}

\maketitle

The classical criterion for the least point of the essential
spectrum was given by Persson~\cite{P} for a Schr\"odinger
operator
$$
-\Delta + q(\x),\q \x\in\Omega,
$$
with the only singularity being at infinity, assuming Dirichlet
boundary conditions on $\Omega$ and assuming $q$ to be bounded
below at infinity. For $q$ bounded below at infinity and near
$\partial \Omega$, Edmunds and Evans~\cite{EE2} extended this
result to include singularities on the boundary $\partial\Omega$
showing that ``if $q\in L^2_{loc}(\Omega)$ and the negative part
of $q$ behaves itself locally, then the essential spectrum" of the
Friedrichs extension of the operator ``is only influenced by the
behaviour of $q$ at $\partial\Omega$ and at infinity in the
respective cases." Conditions (\ref{EEEE}) and (\ref{Boundb2})
below give a mathematical description of the requirement that
``$q$ behaves itself locally." Related techniques were used in
\cite{L} to establish conditions for a purely discrete spectra of
second order elliptic differential operators in weighted $L^2$
spaces including mixed boundary conditions. While still assuming
that $q$ is bounded below near singularities on $\partial\Omega$
or at $\infty$, Evans and Lewis~\cite{EL} used techniques
developed in \cite{EE2} to study even-order elliptic differential
operators in weighted $L_w^2(\Omega)$ spaces with emphasis upon
the criteria for the finiteness or infiniteness of the eigenvalues
below the essential spectrum. We refer to that paper for many
other references to related work.

Edmunds and Evans~\cite{EE3} study the Neumann operator generated
by the degenerate elliptic operator
$$
-div  (d(\x)^{2\mu}\grad\ )+d(\x)^{-2\theta},\q \mu, \theta\ge 0,
$$
on a proper open subset $\Omega\subset\R^n$ where $d(\x):=
dist(\x,\partial\Omega)$. They present upper and lower estimates
for the eigenvalue counting function as well as examining the
embedding properties for associated spaces.

In this paper we study second-order elliptic sesquilinear forms
that give rise to differential operators whose coefficients may
``blow-up" near parts of $\partial\Omega$ including some cases in
which the potential diverges to negative infinity near the
boundary. Applications are given when the coefficients are
approximated by the distance function $d(\x)$ near
$\partial\Omega$.

We follow and abbreviate the structure established in \cite{EL},
but without the introduction of weights or higher-order cases.
Those extensions should be clear from \cite{EL} and the
presentation in this paper.

\section{Introduction}

Let $\Omega\subset \R^n$ be open and connected. Throughout this
paper $\|u\|:=\|u\|_{L^2(\Omega)}$. If $\Omega$ is unbounded then
$\infty$ is considered to be on the boundary  of $\Omega$ in the
sense of a one-point compactification of $\R^n$. The finite points
of the boundary are denoted by $\partial \Omega$. Outside some set
$S$, which contains the singular part of $\partial\Omega$, we assume
that $\partial \Omega$ has a normal in order that certain boundary
conditions are met. If $\Omega$ is unbounded then
$\{\infty\}\subseteq S$, but the emphasis here is upon the part of
$S$ on $\partial\Omega$. The finite part of the singular set
$S\setminus \{\infty\}$ is assumed to be a closed subset of
$\partial\Omega$. Let the singular and regular parts of the boundary
be defined by
$$
\Gamma_S:= N_S\cap\partial\Omega \ \ \ \ \text{and}\ \ \
\Gamma_R:=\partial\Omega\setminus \Gamma_S
$$
where $N_S$ is an open neighborhood of $S\setminus \{\infty\}$ and
$N_\infty:=\{\x:|\x|>K\}$ for some large $K$. We may assume that
$N_S\cap N_\infty=\emptyset$ for unbounded domains $\Omega$.

For an Hermitian matrix $A(\x)=(a_{ij}(\x))$, real-valued $q(\x)$,
$\x\in\Omega$, and $\sigma(s)$, $s\in\Gamma_R$, and a function
$c(s)$ that assumes either the value $1$ or $0$ for $s\in
\Gamma_R$, we are interested in differential operators of the form
$T:\mathcal{D}(T)\to L^2(\Omega)$ with
$$
Tu= \left [-\sum_{i,j=1}^n\frac{\partial}{\partial x_j}\left
(a_{ij}(\x)\frac{\partial}{\partial x_i}\right ) +q(\x)\right ],\
\ \x\in\Omega,
$$
for
$$\begin{array}{rl}
\mathcal{D}(T):=&\{u:u=\varphi \restriction_{_{\Omega}},\ \varphi\in\Con(\R^n\setminus \Gamma_S),\ Tu\in L^2(\Omega),\\
&\ \  \text{and} \
c(s)\frac{\partial\varphi(s)}{\partial\eta_A}+\sigma(s)\varphi(s)=0,\
s\in\Gamma_R\}
\end{array}
$$
where $\partial \varphi/\partial\eta_A:= <A\eta,\grad\varphi>$ and
$\eta$ is the unit outward normal on $\Gamma_R$. The coefficients
$c(s)$ and $\sigma(s)$ are assumed not to be simultaneously zero
allowing for mixed boundary conditions on
$\Gamma_R$. The case $\Gamma_S=\partial\Omega$, which requires
Dirichlet boundary conditions, is included.

In the case of sufficiently smooth coefficients for a symmetric
operator $T$ that is bounded from below, the sesquilinear form
\begin{equation}\label{FormT}
\ft[u,v]:=(Tu,v),\q \mathcal{D}(\ft):=\mathcal{D}(T),
\end{equation}
is closable, Kato~\cite{K}, Theorem VI.1.27, p.318. In the absence
of smooth coefficients, the problem can be interpreted in a weak or
variational sense initially involving only a sesquilinear form. In
that case consider the form
\begin{equation}\label{form}
\ft[u,v]:= \int_{\Omega}\left [<A(\x)\grad u,\grad
v>+qu\overline{v}\right ]d\x
+\int_{\Gamma_R}\sigma(s)u(s)\overline{v(s)}ds
\end{equation}
with domain
$$
\mathcal{D}(\ft):=\{u: u=\varphi\restriction_{_\Omega}, \varphi\in
\Con(\R^n\setminus \Gamma_S)\}.
$$
The value of $c(s)$ is implicit in (\ref{form}). At points where
$\sigma(s)=0$ Neumann conditions are implied so that $c(s)=1$
and at points where $\sigma(s)\neq 0$ there are either Dirichlet
or mixed conditions. For example see R.E.~Showalter~\cite{RES},
Chapter III, Theorem~3A and Example~4.1.

We will give conditions which guarantee that the form is bounded
below and closable. In that case the First Representation Theorem
(Kato~\cite{K}, \S VI, Theorem 2.1) guarantees a unique
self-adjoint operator $\tilde T$ associated with the closure
$\tilde {\ft}$ of $\ft$ for which $\mathcal{D}(\tilde T)\subset
\mathcal{D}(\tilde \ft)$. For forms defined by (\ref{FormT}),
$\tilde T$ is the Friedrichs extension of $T$. Once we have
established that $\ft$ is bounded below and closable, we will
assume that $\ft[u]\ge\|u\|^2$, which can be accomplished by the
addition of a positive constant to $\tilde T$ merely translating
$\sigma_e(\tilde T)$. In this case, according to the Second
Representation Theorem \cite{K}, Theorem VI-2.23, $\tilde
T^{\frac12}$ exists, $\mathcal{D}(\tilde T^{\frac12})=\mathcal
D(\tilde {\ft})$, and
\begin{equation}\label{tildet}
\tilde \ft[u,v]=(\tilde{T}^\frac12u,\tilde{T}^\frac12
v):=(u,v)_{\tilde\ft}.
\end{equation}
In this paper, we will use the Sobolev space $H^1(G)=W^{1,2}(G)$
for an open set $G\subset \R^n$, see Lieb and Loss~\cite{LL},
chapter 7.

Let $\Omega_k$, $k=1,2,\dots$, be bounded domains in $\R^n$ which
satisfy
\begin{itemize}
\item [(i)] $\Omega_k\Subset \Omega_{k+1}$; \item [(ii)]
$\overline{\Omega}\setminus S =\cup_{k=1}^\infty
(\overline{\Omega}\cap\overline{\Omega_k})$; \item [(iii)] there
is a $k_0\in\mathbb{N}$ such that
\begin{equation}\label{k0}
\overline{\Omega}\setminus \Omega_k\subset \overline{\Omega}\cap
(N_S\cup N_\infty)
\end{equation}
for all $k\ge k_0$; and
\item [(iv)] the embedding
$H^1(\Omega_k)\to L^2(\Omega_k)$ is compact for each
$k\in\mathbb{N}$.

\end{itemize}
(Recall the notation $\Omega_k\Subset \Omega_{k+1}$ indicates that
$\Omega_k$ is compactly contained in $\Omega_{k+1}$, i.e.
$\overline{\Omega}_{k}$ is compact and
$\overline{\Omega}_k\subset\Omega_{k+1}$.) This family of domains
$\{\Omega_k\}_{k=1}^\infty$ is an \emph{S-admissible family of
domains in} $\Omega$ as defined in Edmunds and Evans~\cite{EE},
p.278. Note that (iv) holds provided the Rellich embedding theorem
applies, e.g., if $\partial (\Omega\cap\Omega_k)$ has the segment
property, Agmon~\cite{A}, Theorem 3.8. In most applications
considerable flexibility in constructing each $\Omega_k$ will be
available.

Denote the maximum and minimum eigenvalue of $A(\x)$ by $\nu_A(\x)$
and $\mu_A(\x)$ respectively. The notation $f_-(\x):=
-\min\{f(\x),0\}$ and $f_+(\x):=f(\x)+f_-(\x)$ will be used. Assume
\newline
\newline \emph{Hypothesis (H)}: For each $k$, assume that
\begin{itemize}
\item[(a)] $\partial(\Omega\cap\Omega_k)$ is $C^1$; \item[(b)]
$\mu_A(\x)>0$ a.e. on $\Omega$ and $\mu_A^{-1}\in
L^\infty(\Omega\cap\Omega_k)$; \item[(c)] $q\in
L^\alpha(\Omega\cap\Omega_k)$ with
\begin{equation}\label{EEEE}
 \alpha\left\{\begin{array}{clc} =&\frac{n}{2},& n>2,\\
>&1,& n=2;\end{array}\right.
\end{equation}
\item[(d)] $\sigma_-(s)=0$ for $s\in \Gamma_R\setminus
\overline{\Omega}_{k_0}$; and \item [(e)] $\sigma\in
L^\beta(\Gamma_R)$ with
$$
\beta \left\{ \begin{array}{rlc} =&n-1,& n>2,\\ >&1,&
n=2.\end{array}\right.
$$
\end{itemize}

The next lemma is a special case of Lemma~1 of \cite{EL}. We refer
to that paper for the complete proof. It indicates the degree of
unbounded behavior of $q_-$ that is allowed locally.
\begin{Lemma}\label{LemA}
If (H) holds, then for $\epsilon>0$ and each $k\in\mathbb{N}$
there is a $K(\epsilon, k)>0$ such that
\begin{equation}\label{Boundb2}\begin{array}{rl}
\int_{\Omega\cap\Omega_k}q_-|u|^2d\x&+\int_{\Gamma_R\cap\overline{\Omega}_k}|\sigma(s)||u(s)|^2ds\\
\le& \epsilon \int_{\Omega\cap\Omega_k}<A\grad u,\grad
u>d\x+K(\epsilon,k)\int_{\Omega\cap\Omega_k}|u|^2d\x
\end{array}
\end{equation}
for all $u\in\mathcal{D}(\ft)$.
\end{Lemma}
\begin{proof} The proof follows from the Monotone Convergent Theorem, the H\"older Inequality, and the Sobolev Inequality.
\end{proof}

\section{The main results}

When we know of the existence of $\tilde T$ we let
$\ell_e=\ell_e(\tilde T)$ denote the least point of its essential
spectrum. The following Proposition compares with Corollary~7D,
Chapter III, of R.E.~Showalter~\cite{RES}.

\begin{Proposition}\label{Main}
Assume hypothesis (H), that
$$
\nu_A(\x)\in L^\infty(\Omega\cap\Omega_k),\q k\in\mathbb{N},
$$
and that for all $k$ sufficiently large
\begin{equation}\label{Coercive}
\ft[u]+ \alpha_k \|u\|_{L^2(\Omega\cap\Omega_k)}^2\ge
c_k\|u\|_{H^1(\Omega\cap\Omega_k)}^2,\q u\in\mathcal{D}(\ft),
\end{equation}
for positive constants $\alpha_k$ and $c_k$.

If $\ft$ is bounded below and closable, then
\begin{equation}\label{lsube}\begin{array}{rl}
\ell_e:=& \inf\{\lambda:\lambda\in\sigma_e(\tilde{T})\}\\
=&\underset{k\to\infty}{\lim}\
\underset{\|u\|=1}{\inf}\left\{\ft[u]:u\in\mathcal{D}(\ft),
\text{supp\ }u\subset\Omega\setminus\Omega_k\right\}.
\end{array}
\end{equation}
\end{Proposition}
\begin{proof}  It will suffice to show that the following holds (see p.476 of \cite{EE}):
\begin{itemize}
\item [($\mathcal{A}$)] For each $k\in \mathbb{N}$ large enough
and $\phi\in \Con(\R^n\setminus \Gamma_S)$ such that
\begin{equation}\label{CutoffconditionAaaB}
\phi(\x)=\left\{\begin{array}{cc} 1, & \x\in \Omega_k,\\ 0, &
\x\notin \Omega_{k+1},   \end{array}\right.
\end{equation}
with $0\le \phi \le 1$, we have
\begin{itemize}
\item [(i)] $\phi v\in\mathcal D(\ft)$ for every $v\in
\mathcal{D}(\ft)$ and \item [(ii)] if $v_\ell\in \mathcal{D}(\ft)$
with $\|v_\ell\|_{\tilde \ft}=1$ and $v_\ell \rightharpoonup 0$ in
the Hilbert space $H(\tilde{\ft}):=(\mathcal{D}(\tilde
\ft);\|\cdot\|_{\tilde \ft})$, then
$$
\|(1-\phi)v_\ell\|_{\tilde \ft}^2\le 1+o(1)\ \ \text{as
$\ell\to\infty$}.
$$

\end{itemize}

\end{itemize}
Part $(\mathcal{A})$(i) is immediate.

Since $\ft$ is bounded below, without loss of generality, we may
assume that $\ft\ge 1$ on $\mathcal{D}(\ft)$ as discussed above.
Therefore (\ref{tildet}) holds.

For all $u\in\mathcal{D}(\ft)$ and any $\phi$ satisfying
(\ref{CutoffconditionAaaB})
$$\begin{array}{rl}
\|(1-& \phi)u\|_{\tilde \ft}^2-\int_{\Gamma_R}\sigma |(1-\phi)u|^2ds\\
=&\int_{\Omega\setminus\Omega_k}\left\{<A\grad (1-\phi)u,\grad (1-\phi)u>+q|(1-\phi)u|^2\right\}dx\\
= &\int_{\Omega\setminus \Omega_k} \left \{(1-\phi)^2<A\grad u,\grad u> +q|u|^2-(2-\phi)\phi\ q|u|^2\right \}d\x\\
&2\int_{(\Omega\cap\Omega_{k+1})\setminus \Omega_k}\text{Re}\left \{<A^{1/2}(1-\phi)\grad u,A^{1/2}u\grad (1-\phi)>\right\}\\
&+\int_{(\Omega\cap\Omega_{k+1})\setminus\Omega_k}<A\grad\phi,\grad\phi>|u|^2d\x\\
\le &\int_{\Omega\setminus \Omega_k} \left \{(1-\phi)^2<A\grad u,\grad u> +q|u|^2+(2-\phi)\phi\ q_-|u|^2\right \}d\x\\
&-2\int_{(\Omega\cap\Omega_{k+1})\setminus \Omega_k}(1-\phi)\text{Re}\left \{<A^{1/2}\grad u,A^{1/2}\grad \phi>\overline{u}\right\}\\
&+\int_{(\Omega\cap\Omega_{k+1})\setminus\Omega_k}<A\grad\phi,\grad\phi>|u|^2d\x\\
\le &\int_{\Omega} \left [<A\grad u,\grad u>+q|u|^2\right ]d\x +\int_{\Omega\cap\Omega_{k+1}} q_-|u|^2d\x\\
&+\delta\int_{(\Omega\cap\Omega_{k+1})\setminus \Omega_k}<A\grad u,\grad u>d\x\\
&+(1+\delta^{-1}) \int_{(\Omega\cap\Omega_{k+1})\setminus \Omega_k}<A\grad\phi,\grad\phi>|u|^2d\x\\
\end{array}
$$
for $\delta>0$. Similarly,
$$
\int_{\Gamma_R}\sigma
(1-\phi)^2|u|^2ds=\int_{\Gamma_R}\sigma|u|^2ds-\int_{\Gamma_R\cap
\overline{\Omega}_{k+1}}\phi(2-\phi)\sigma|u|^2ds.
$$
Since $\nu_A\in L^\infty(\Omega\cap\Omega_k)$ and (\ref{Boundb2})
holds for each $k$, it then follows that
$$\begin{array}{rl}
\|(1-\phi)u\|_{\tilde \ft}^2 \le \|u\|^2_{\tilde \ft} &+\delta'
\int_{\Omega\cap\Omega_{k+1}}<A\grad u,\grad u> d\x\\
& +C(\delta',k)
\int_{\Omega\cap\Omega_{k+1}}|u|^2d\x,\\
\end{array}
$$
for an arbitrarily small $\delta'$ and $C(\delta',k)>0$. According
to the coercivity requirement (\ref{Coercive}) and the fact that
$\nu_A\in L^\infty(\Omega\cap\Omega_k)$ for each $k\in\mathbb{N}$
$$
\|(1-\phi)u\|_{\tilde \ft}\le (1+\frac{\epsilon}{c_{k+1}})
\|u\|_{\tilde
\ft}^2+(C(\delta',k)+\frac{\alpha_k}{c_{k+1}})\|u\|^2_{L^2(\Omega\cap
\Omega_{k+1})}
$$
for an arbitrarily small $\epsilon$.

As in $(\mathcal{A})$(ii) suppose that $\{v_\ell\}\subset
\mathcal{D}(\ft)$ satisfies $\|v_\ell\|_{\tilde t}=1$ and
$v_\ell\rightharpoonup 0$ in $H(\tilde \ft)$. We have that
$$
\|(1-\phi)v_\ell\|_{\tilde t}^2\le
1+\epsilon'+C'(\delta',k)\|v_\ell\|^2_{L^2(\Omega\cap
\Omega_{k+1})}.
$$

By (\ref{Coercive}) and the fact that $\ft\ge 1$, it follows that
the embedding $H(\tilde \ft)\to H^1(\Omega\cap\Omega_{k+1})$ is
continuous. Since $H^1(\Omega\cap\Omega_{k+1})\to
L^2(\Omega\cap\Omega_{k+1})$ is compact, then
$\|v_\ell\|^2_{L^2(\Omega\cap\Omega_{k+1})}= o(1)$ as
$\ell\to\infty$. Hence,
$$
\|(1-\phi)v_\ell\|_{\tilde t}^2\le 1+o(1)
$$
since $\epsilon'$ can be chosen arbitrarily small. That completes
the proof.

\end{proof}

In unbounded domains $\Omega$ we will assume that $q$ is bounded
below at infinity as in (\ref{infty}) below. When we know a priori
that $\ft$ is bounded below, we may assume without loss of
generality that for $k$ sufficiently large $q(\x)>0$ for $\x\in
(\Omega\setminus\Omega_k)\cap N_\infty$ as well as $\ft[u]\ge
\|u\|^2$, mentioned above, since the addition of a constant only
translates the spectrum.

In contrast to \cite{EL}, \cite{EE2}, and the classical criterion of
Persson~\cite{P}, we are not requiring that the potential $q$ be
bounded below in a neighborhood $N_S$ of the finite singularities.
The next theorem shows that in the case of a coefficient degenerate
on $S\cap\partial\Omega,$ the existence of a Hardy-type inequality
in a neighborhood of the singularities may be sufficient to ensure
that the form is closable and bounded below, i.e., inequality
(\ref{Bdd1}) replaces the requirement that $q$ be bounded below on
$\partial\Omega$.

\begin{Theorem}\label{closable}
Assume (H) holds and that for some $\gamma\in (0,1)$ and $k_0$
given in (\ref{k0})
\begin{equation}\label{Bdd1}
\int_{(\Omega\setminus \Omega_{k})\cap N_S}[(1-\gamma)<A\grad
u,\grad u>-q_-|u|^2]d\x\ge 0,\q u\in\mathcal{D}(\ft),
\end{equation}
for all $k\ge k_0$ and
\begin{equation}\label{infty}
\underset{k\to\infty}{\lim} \underset{\x\in
(\Omega\setminus{\Omega_k})\cap N_\infty}{ess\ \sup}q_-(\x)
=C_\infty<\infty
\end{equation}
when $\Omega$ is unbounded. Then $\ft$ is bounded below and
closable and (\ref{Coercive}) holds. Furthermore, if
$$
\nu_A\in L^\infty(\Omega\cap\Omega_k), \q k\in\mathbb{N},
$$
then $\ell_e(\tilde T)$ is given by (\ref{lsube}).
\end{Theorem}
\begin{proof}
We give the proof in the case that $\Omega$ is unbounded. The
proof for $\Omega$ bounded requires only slight modification.

Let
$$\begin{array}{rl}
\mathfrak{t}_1[u]:=&\int_\Omega [<A\grad u,\grad u>+q_+|u|^2]d\x+\int_{\Gamma_R}\sigma_+(s)|u(s)|^2ds,\\
\mathfrak{t}_1'[u]:=&-\int_\Omega q_-|u|^2d\x
-\int_{\Gamma_R}\sigma_-(s)|u(s)|^2ds
\end{array}
$$
with $\mathcal{D}(\ft)=\mathcal{D}(\ft_1')=\mathcal{D}(\ft_1)$ and
$\ft=\ft_1+\ft_1'$.

We first show that $\mathfrak{t}_1'$ is $\mathfrak{t}_1$-bounded
with $\mathfrak{t}_1$-bound less than 1. Then, in order to
conclude that $\ft$ is closable it will suffice to show that
$\mathfrak{t}_1$ is closable - see Kato~\cite{K}, Theorem 1.33,
p.320.

Let $k\ge k_0$ in (\ref{Boundb2}) recalling that $\sigma_-(s)=0$
for $s\in \Gamma_R\setminus \overline{\Omega}_{k_0}$ according to
(\emph{H}). Without loss of generality, we may assume that for
$\delta>0$,
$$
q_-(\x)<C_\infty+\delta,\q \x\in (\Omega\setminus
\Omega_{k_0})\cap N_\infty.
$$
Then it follows from (\ref{Bdd1}) and (\ref{Boundb2}) that for all
$u\in\mathcal{D}(\ft)$, $\epsilon\le (1-\gamma)$, and
$\alpha(\epsilon,k)\ge \max\{K(\epsilon,k), C_\infty+\delta\}+1$,
\begin{equation}\label{tboundedbB}\begin{array}{rl}
|\ft_1'[u]|\le& (1-\gamma) \int_{(\Omega\setminus\Omega_{k})\cap
N_S}<A(\x)\grad u,\grad u>d\x\\
&+\int_{(\Omega\setminus\Omega_{k})\cap N_\infty}q_-|u|^2d\x
+\int_{\Omega\cap\Omega_{k}}q_-|u|^2d\x+\int_{\Gamma_R}\sigma_-|u|^2ds\\
\le& (1-\gamma) \ft_1[u]
+(C_\infty+\delta)\int_{(\Omega\setminus\Omega_{k})\cap N_\infty}|u|^2d\x\\
&+K(\epsilon,k)\int_{\Omega\cap\Omega_{k}}|u|^2d\x\\
\le& (1-\gamma) \ft_1[u] +\alpha(\epsilon,k)\|u\|_{L^2(\Omega)}^2.
\end{array}
\end{equation}
Therefore, $\ft_1'$ has $\ft_1$-bound less than $1$.

Note that (\ref{tboundedbB}) implies the inequality
\begin{equation}\label{bddbelow}\begin{array}{rl}
\ft[u]+\alpha(\epsilon,k)\|u\|_{L^2(\Omega)}^2 \ge \gamma \ft_1[u]
\ge 0.
\end{array}
\end{equation}
Therefore, $\ft$ is bounded below.

To show that $\ft_1$ is closable in $L^2(\Omega)$, choose
$\{\varphi_n\}\subset \mathcal{D}(\ft)$ such that
\begin{equation}\label{t-convergentAAbB}
\ft_1[\varphi_n-\varphi_m]\to 0,\ \ \ \|\varphi_n\|\to 0\ \ \ \
\text{as}\ \ m,n\to\infty,
\end{equation}
i.e., $\{\varphi_n\}$ is $\ft_1$-convergent to $0$. Then, we must
show that $\ft_1[\varphi_n]\to 0$ as $n\to\infty$. First, note
that (\ref{t-convergentAAbB}) implies that
$$
\int_{\Omega}<A(\x)\grad
(\varphi_n-\varphi_m),\grad(\varphi_n-\varphi_m)>d\x\to 0,\q
m,n\to\infty.
$$
It follows as in (3.13) of \cite{EL} that
\begin{equation}\label{AconvergenceB}
\int_{\Omega}<A(\x)\grad \varphi_n,\grad\varphi_n>d\x\to 0,\q
n\to\infty.
\end{equation}

Since
$$
\ft_1[u] +\alpha(\epsilon,k_0)\|u\|^2\ge \int_\Omega q_+|u|^2d\x +
\|u\|^2
$$
then $\{\varphi_n\}$ must be a Cauchy sequence in
$L^2_{q_++1}(\Omega)$. Since this space is complete, we must have
that $\varphi_n\to\psi$ for some $\psi\in L^2_{q_++1}(\Omega)$.
But, $\varphi_n\to 0$ in $L^2(\Omega)$ implies that $\psi\equiv
0$.

Consequently, we have shown that
$$
\int_\Omega [<A\grad \varphi_n,\grad
\varphi_n>+q_+|\varphi_n|^2]d\x\to 0.
$$
We need to show that
$$
\int_{\Gamma_R}\sigma_+(s)|\varphi_n(s)|^2ds \to 0
$$
in order to complete the proof. An analysis similar to (3.17) in
\cite{EL} applies here as well since
$$
\ft_1[(\varphi_n-\varphi_m)]\ge
\int_{\Gamma_R}\sigma_+(s)|(\varphi_n-\varphi_m)|^2ds\ge 0.
$$
Hence, $\{\varphi_n\}$ is Cauchy in $L^2_{\sigma_++1}(\Gamma_R)$
and converges to a $v\in L^2_{\sigma_++1}(\Gamma_R)$. By
(\ref{Boundb2}) and (\ref{AconvergenceB}) we conclude that $v=0$
on $\Gamma_R\cap\overline{\Omega}_k$ for each $k$. Since
$\Gamma_R=\cup_k(\Gamma_R\cap\overline{\Omega}_k)$, then $v=0$ on
$\Gamma_R$ which is what we wanted to show.

Therefore, $\ft$ is bounded below and closable. As discussed
above, it will suffice for the remainder of the proof to assume
that $\ft\ge 1$ and $q >0$ in $(\Omega\setminus\Omega_{k})\cap
N_\infty$ for $k$ large.

Since $q_-(\x)=0$ for $k$ large and $\x\in (\Omega\setminus
\Omega_k)\cap N_\infty$, it follows from (\ref{tboundedbB}) that
$$
|\ft_1'[u]|\le (1-\gamma) \ft_1[u]
+K(\epsilon,k)\int_{\Omega\cap\Omega_{k}}|u|^2d\x,\q
u\in\mathcal{D}(\ft),
$$
which implies that
$$
\ft[u]+K(\epsilon,k)\|u\|_{L^2(\Omega\cap\Omega_k)}^2 \ge \gamma
\ft_1[u] \ge \gamma \int_{\Omega\cap\Omega_k}<A\grad u,\grad
u>d\x,\ \  u\in\mathcal{D}(\ft).
$$
Since $\mu_A^{-1}\in L^\infty(\Omega\cap\Omega_k)$, then
(\ref{Coercive}) holds. If we know that
$$
\nu_A\in L^\infty(\Omega\cap\Omega_k), \q k\in\mathbb{N},
$$
then it follows from Proposition~\ref{Main} that (\ref{lsube})
holds.
\end{proof}

Note that if $q$ is bounded below by $B<0$ on
$(\Omega\setminus\Omega_{k_0})\cap N_S$ as assumed in earlier work,
e.g., \cite{EE2}, \cite{EE}, and \cite{EL}, then we may apply
Theorem~\ref{closable} to the form
$\ft[u]+|B|\int_{\Omega}\chi_{_{(\Omega\setminus\Omega_{k_0})\cap
N_S}}|u|^2d\x$.

It may be advantageous to need only show that the inequality in
(\ref{Bdd1}) holds for $u\in H_0^1((\Omega\setminus\Omega_{k_0})\cap
N_S)$. The next Theorem shows that is allowed. However, we will see
in the applications below that in some cases it is best to use
(\ref{Bdd1}) directly avoiding certain convexity requirements.

\begin{Theorem} \label{main2}
Assume hypothesis (H), that
$$
\nu_A\in L^\infty(\Omega\cap\Omega_k),\q k\in\mathbb{N},
$$
and (\ref{infty}) for $\Omega$ not bounded. If for all $\varphi\in
H_0^1((\Omega\setminus\Omega_{k_0})\cap N_S)$
\begin{equation}\label{newfact}
\int_{(\Omega\setminus \Omega_{k_0})\cap N_S}
[(1-\gamma)<A(\x)\grad \varphi,\grad \varphi >
-q_-(\x)|\varphi|^2]d\x\ge 0
\end{equation}
for some $\gamma\in (0,1)$, then (\ref{lsube}) holds.

\end{Theorem}
\begin{proof}

Recall that $N_S$ is an open neighborhood of the finite
singularities, $S\setminus\{\infty\}$, with
$\overline{\Omega}\setminus\Omega_k\subset \overline{\Omega}\cap
(N_S\cup N_\infty)$ for $k\ge k_0$. We employ a simple IMS
localization formula - see \cite{Cyconetal}, p.28. Choose
$k_2>k_1\ge k_0$. There exists $\phi_1\in C^\infty(\R^n)$ for
which
$$
\phi_1(\x)=\left\{\begin{array}{lc} 1,& \x\in
(\overline{\Omega}\setminus\Omega_{k_2})\cap N_S,\\ 0,&
\x\in\overline{\Omega}\cap\Omega_{k_1},\end{array}\right.
$$
(with the support of $\phi$ extending into $\R^n\setminus
\overline{\Omega}$ as needed) and $\phi_2$ such that
\begin{itemize}
\item $\phi_j(\x)\in [0,1]$ for $j=1,2,$ and all $\x\in\R^n$;
\item $\phi_1^2(\x)+\phi_2^2(\x)\equiv 1$ for all $\x\in\R^n$;
\item $\phi_j\in C^\infty(\R^n)$; and \item
$\sup_{\x\in\R^n}[|\grad\phi_1(\x)|^2+|\grad
\phi_2(\x)|^2]<\infty$.
\end{itemize}
Recall the pointwise identity that gives rise to the IMS
localization formula: for $u\in\mathcal{D}(\ft)$ and $j=1,2$,
\begin{equation}\label{IMS}\begin{array}{rl}
<A\grad (\phi_j u),\grad (\phi_j u)> =&\phi_j^2<A\grad u,\grad u> +<A\grad\phi_j,\grad\phi_j>|u|^2\\
&+ \Re{e} <A\grad \phi_j^2,\overline{u}\grad u>.
\end{array}
\end{equation}
Summing over $j=1,2$, and integrating yields the identity
$$\begin{array}{rl}
\ft[u]=&\sum_{j=1}^2\int_\Omega [<A(\x)\grad (\phi_j u),\grad (\phi_j u)>+q|\phi_ju|^2-<A\grad \phi_j,\grad\phi_j>|u|^2]d\x\\
&+\int_{\Gamma_R}\sigma(s)|u(s)|^2ds
\end{array}
$$
since $\phi_2(s)=1$ on $\Gamma_R$. Then
$\phi_1u\in\Con(\Omega\setminus\Omega_{k_1})$.

 It follows from the pointwise identity (\ref{IMS}) that
$$\begin{array}{rl}
\int_{(\Omega\setminus\Omega_{k_1})\cap N_S}&<A\grad u,\grad u>d\x\\
=&\sum_{j=1}^2\int_{(\Omega\setminus\Omega_{k_1})\cap N_S} [<A\grad (\phi_j u),\grad (\phi_j u)>-<A\grad \phi_j,\grad\phi_j>|u|^2]d\x\\
\ge &\int_{(\Omega\setminus \Omega_{k_1})\cap N_S} <A(\x)\grad (\phi_1 u),\grad (\phi_1 u)>d\x\\
&-C_{k_2}\int_{(\Omega_{k_2}\setminus \Omega_{k_1})\cap N_S}|u|^2d\x\\
\end{array}
$$
for
$$
C_{k_2}:=\sup_{\x\in (\Omega_{k_2}\setminus\Omega_{k_1})\cap
N_S}\sum_{j=1}^2<A\grad \phi_j,\grad\phi_j>\ <\infty
$$
since $\nu_A\in L^\infty(\Omega\cap\Omega_k)$, $k\in \mathbb{N}$.

Since (\ref{newfact}) holds for $\gamma\in (0,1)$,
$$\begin{array}{rl}
\int_{(\Omega\setminus\Omega_{k_1})\cap N_S}<A\grad u,\grad u>d\x
&+C_{k_2}\int_{(\Omega_{k_2}\setminus \Omega_{k_1})\cap N_S}|u|^2d\x\\
\ge& (1-\gamma)^{-1} \int_{(\Omega\setminus\Omega_{k_1})\cap N_S}q_-|\phi_1u|^2d\x\\
= & (1-\gamma)^{-1}[\int_{(\Omega\setminus\Omega_{k_1})\cap N_S}q_-|u|^2d\x\\
&-\int_{(\Omega_{k_2}\setminus\Omega_{k_1})\cap
N_S}q_-|\phi_2u|^2d\x].
\end{array}
$$
As in Lemma~\ref{LemA} we have that for any $\epsilon>0$ there is a
positive constant $K(\epsilon,k_2)$ such that
$$\begin{array}{rl}
\int_{(\Omega_{k_2}\setminus\Omega_{k_1})\cap N_S}q_-|\phi_2u|^2d\x]
\le &\epsilon\int_{(\Omega_{k_2}\setminus\Omega_{k_1})\cap
N_S}<A\grad u,\grad
u>d\x\\
&+K(\epsilon,k_2)\int_{(\Omega_{k_2}\setminus\Omega_{k_1})\cap
N_S}|u|^2d\x
\end{array}
$$
(see (2.9) of \cite{EL}) which implies that
$$\begin{array}{rl}
(1+\epsilon)&\int_{(\Omega\setminus\Omega_{k_1})\cap N_S}<A\grad
u,\grad u>d\x +C(\epsilon,k_2)\int_{(\Omega_{k_2}\setminus
\Omega_{k_1})\cap N_S}|u|^2d\x\\
& \ge (1-\gamma)^{-1}\int_{(\Omega\setminus\Omega_{k_1})\cap
N_S}q_-|u|^2d\x
\end{array}
$$
for $C(\epsilon,k_2):=C_{k_2}+K(\epsilon,k_2)$. Then, for $\epsilon$
chosen sufficiently small $(1+\epsilon)(1-\gamma)\in (0,1)$. Since
$k_1$ is an arbitrary integer greater than or equal to $k_0$, the
hypothesis of Theorem~\ref{closable} holds for
$$
\mathfrak{h}[u]:=\ft[u]-(1-\gamma)C(\epsilon,k_2)
\int_\Omega\chi_{_{(\Omega_{k_2}\setminus \Omega_{k_1})\cap
N_S}}|u|^2d\x,\q u\in\mathcal{D}(\ft),
$$
implying that $\mathfrak{h}[u]$ is bounded below and closable and,
as shown in the proof of Theorem~\ref{closable}, that
(\ref{Coercive}) holds for $\mathfrak{h}$. But, this implies that
$\ft$ is bounded below and closable (cf. (\ref{t-convergentAAbB}))
and (\ref{Coercive}) holds as well for $\ft$. The conclusion
follows from Proposition~\ref{Main}.
\end{proof}

With appropriate conditions required of the coefficients, inequality
(\ref{newfact}) is associated with the existence of a nonnegative
solution of the Dirichlet problem for
$$
-(1-\gamma)\ div(A(\x)\grad u)-q_-(\x)u=0
$$
on $(\Omega\setminus\Omega_{k_0})\cap N_S$, the absence of nodal
domains, and the finiteness of the negative spectrum (\cite{Al},
\cite{P},\cite{Pi}).

\begin{Corollary}\label{CorThm2}
Assume the hypothesis of Theorem~\ref{main2} and for $k\ge k_0$
define
$$
L_S[u;k]:= \int_{(\Omega\setminus \Omega_{k})\cap N_S}[<\gamma
A(\x)\grad u,\grad u>+q_+(\x)|u|^2]d\x,\q u\in\mathcal{D}(\ft).
$$
Then, for $\Omega$ bounded
$$\begin{array}{rl}
\ell_e \ge&\lim_{k\to\infty}\underset{\|u\|=1}{\inf}L_S[u;k]
\end{array}
$$
with the infimum taken over all $u\in\mathcal{D}(\ft)$ with $supp\ u
\subset (\Omega\setminus\Omega_k)\cap N_S$. If $\Omega$ is
unbounded and (\ref{infty})  holds, then
$$\begin{array}{rl}
\ell_e
\ge&\lim_{k\to\infty}\underset{\|u\|=1}{\inf}L_S[u;k]-C_\infty.
\end{array}
$$
\end{Corollary}
\begin{proof} We give the proof for the case in which $\Omega$ is unbounded. The adaptation for $\Omega$ bounded is straightforward.
According to Theorem~\ref{main2}, for $k\ge k_0$ and $\varphi :=
u/\|u\|$ for $u\in \mathcal{D}(\ft)$ with $supp\ u\subset
\Omega\setminus\Omega_k$
$$\begin{array}{rl}
\ft[\varphi]\ge&\gamma\int_{(\Omega\setminus\Omega_k)\cap
N_S}<A(\x)\grad \varphi,\grad\varphi>d\x
+\int_{(\Omega\setminus\Omega_{k})\cap N_S}q_+(\x)|\varphi|^2d\x\\
&+\int_{(\Omega\setminus\Omega_{k})\cap N_\infty}[<A(\x)\grad\varphi,\grad\varphi>+q(\x)|\varphi|^2]d\x\\
\ge&\int_{(\Omega\setminus\Omega_k)\cap N_S}[\gamma <A(\x)\grad
\varphi,\grad\varphi>
+q_+(\x)|\varphi|^2]d\x\\
&-(C_\infty +\delta)\int_{(\Omega\setminus\Omega_{k})\cap
N_\infty}|\varphi|^2d\x
\end{array}
$$
for some small $\delta>0$. Therefore,
$$\begin{array}{rl}
\ell_e =&\lim_{k\to\infty}\inf_{\varphi}\ft[\varphi]\\
\ge&\lim_{k\to\infty}\inf_{\varphi}L_S[\varphi;k]-C_\infty.
\end{array}
$$
\end{proof}

\section{Applications using Hardy inequalities in $d(\x)$.}

In this section we explore applications of Theorems~\ref{closable}
\& \ref{main2} with some of the more recent results for Hardy
inequalities given in terms of the distance to the boundary of the
domain, i.e., $d(\x):=\ dist(\x,\partial\Omega)$. Weighted Hardy
inequalities in $L^2(G)$, which best suit our purposes, are of the
following form: for an open connected set $G\subset\R^n$ and $u\in
H_0^1(G)$
\begin{equation}\label{E1}
\int_G d(\x)^\beta|\grad u(\x)|^2d\x \ge \kappa(\beta)\int_G
\frac{|u(\x)|^2}{d(\x)^{2-\beta}}d\x
+\lambda(G)\int_{G}d(\x)^\alpha|u(\x)|^2d\x
\end{equation}
with $\beta<1$ and $\alpha > (\beta -2)$. Here, $\kappa(\beta)$ is
assumed to be positive for each $\beta<1$ and $\lambda(G)\ge 0$
depends upon certain geometric properties of $G$, e.g., the
diameter of $G$, the volume of $G$, etc. Several results of this
type are discussed below.

\begin{Corollary}\label{CorAppl} Assume hypothesis $(H)$, $\nu_A\in L^\infty(\Omega\cap\Omega_k)$ for all $k$,
and that for some $\beta<1$
\begin{equation}\label{INEQQ}
\mu_A(\x)\ge d(\x)^{\beta},\q \x\in
(\Omega\setminus\Omega_{k_0})\cap N_S.
\end{equation}
For $\Omega$ unbounded assume that $q_-$ is bounded below at
infinity as in (\ref{infty}). Finally, assume that (\ref{E1})
holds for some $\beta<1$ and for
$G=(\Omega\setminus\Omega_{k_0})\cap N_S$. If for some $\gamma\in
(0,1)$
\begin{equation}\label{Qeq2}
q_-(\x)\le
(1-\gamma)[\frac{\kappa(\beta)}{d(\x)^{2-\beta}}+\lambda(G)
d(\x)^\alpha],\q \x\in G,
\end{equation}
then $\ft$ is bounded below and closable and the spectrum of
$\tilde T$ is purely discrete.

\end{Corollary}
\begin{proof} The fact that $\ft$ is bounded below and closable follows from Theorem~\ref{closable}.
By (\ref{E1}) and (\ref{Qeq2}) the hypothesis of
Theorem~\ref{main2} holds. We may apply Corollary~\ref{CorThm2}.
For $k>k_0$
$$
L_S[u;k]\ge \int_{(\Omega\setminus\Omega_k)\cap N_S}\gamma
d(\x)^\beta |\grad u|^2d\x \ge
\gamma\kappa(\beta)\int_{(\Omega\setminus\Omega_k)\cap
N_S}\frac{|u|^2}{d(\x)^{2-\beta}}d\x
$$
according to (\ref{INEQQ}) followed by (\ref{E1}). According to
property (ii) of the S-admissible family of domains
$\{\Omega_k\}_{k=1}^\infty$ we may assume that $d(\x)<\frac{1}{k}$
for $\x\in (\Omega\setminus\Omega_k)\cap N_S$ and $k\ge k_0$. Since the infimum
in Corollary~\ref{CorThm2} is taken over all
$u\in\mathcal{D}(\ft)$ with support in
$(\Omega\setminus\Omega_k)\cap N_S$, then for $k\ge k_0$
$$
 \underset{\| u\|=1}{\inf}L_S[u;k]\ge \gamma\kappa(G)
 k^{2-\beta},\q \beta <1.
$$
Letting $k\to \infty$, we conclude that $\ell_e=\infty$ implying
that the spectrum is purely discrete.
\end{proof}

Corollary~\ref{CorAppl} indicates that if a Hardy inequality
(\ref{E1}) holds, the form $\ft$ can be bounded below and closable
even though all coefficients are degenerate at parts of the boundary
$\partial\Omega$. We review some of the earlier results in which
(\ref{E1}) holds.

For $\alpha=\beta=0$, (\ref{E1}) reduces to
\begin{equation}\label{E1a}
\int_\Omega |\grad u(\x)|^2d\x \ge \frac14\int_\Omega
\frac{|u(\x)|^2}{d(\x)^{2}}d\x
+\lambda(\Omega)\int_{\Omega}|u(\x)|^2d\x.
\end{equation}
Recent results for this inequality were motivated by work of
Brezis and Marcus in \cite{BM} who showed that for $\Omega$ convex
with $\partial\Omega\in C^2$,
$\lambda(\Omega)\ge\frac{1}{4D(\Omega)^2}$ with $D(\Omega)$
denoting the usual diameter of $\Omega$. For the ``interior
diameter" defined by $D_{int}(\Omega):=2\sup_{\x\in\Omega}d(\x)$,
Filippas, Maz'ya, and Tertikas \cite{FMT} showed that for $\Omega$
convex, $\lambda(\Omega)\ge\frac{3}{D_{int}(\Omega)^2}$.
Subsequently, Avkhadiev and Wirths \cite{AW} have shown that
$\lambda(\Omega)\ge \frac{4\lambda_0}{D_{int}(\Omega)^2}$ where
$\lambda_0\ge 0.94$. Using methods of Davies \cite{D},
M.~Hoffmann-Ostenhof, T~Hoffmann-Ostenhof, and Laptev \cite{HHL}
answered a question posed by Brezis and Marcus showing that for
convex domains $\lambda(\Omega)\ge \frac{K(n)}{4|\Omega|^{2/n}}$,
$K(n):=n^{1-2/n}|\bbS^{n-1}|^{2/n}$, in which $|\Omega|$ denotes
the volume of $\Omega$. Using similar methods, Evans and
Lewis~\cite{EL2} showed that $\lambda(\Omega)\ge
\frac{3K(n)}{2|\Omega|^{2/n}}$.

Since a ball of diameter $D_{int}(\Omega)$ must be contained in
$\Omega$, it follows that for $n=2,3$, the results for
$\lambda(\Omega)$ in the paper of Filippas, Maz'ya, and Tertikas
\cite{FMT} are comparable to those in terms of the volume improving
the inequality in the paper of M.~Hoffmann-Ostenhof,
T~Hoffmann-Ostenhof, and Laptev~\cite{HHL}. Also, there is some
advantage in the fact that the inequalities of \cite{AW}, \cite{BM}
and \cite{FMT} do not require $|\Omega|$ to be finite, e.g., $\Omega
= \omega\times \R$ with $\omega\subset\R^{n-1}$ convex. In that case
$|\Omega|=\infty$, but $D_{int}(\Omega)<\infty$ if
$D_{int}(\omega)<\infty$.

While applying some of these inequalities in
Corollary~\ref{CorAppl}, convexity may be required, but that
requirement is diminished by the fact it is needed only on
$(\Omega\setminus\Omega_{k_0})\cap N_S$ and not necessarily on
$\Omega$. In addition, a certain degree of flexibility is available
in constructing the family $\{\Omega_k\}_{k=1}^\infty$ in $N_S$.
Nevertheless, we will also be interested in inequalities not
requiring convexity.

In a domain $\Omega\subset\R^n$ the distance function $d(\x)$ is
uniformly Lipschitz continuous (cf. Gilbarg and Trudinger
\cite{GT}, \S 14.6) and consequently, differentiable almost
everywhere according to Rademacher's theorem. Moreover, if
$\Omega$ is bounded and $\partial\Omega\in C^k$, $k\ge 2$, then
for some $\delta>0$ sufficiently small, $d\in C^k(\Omega_\delta)$
in which $\Omega_\delta:=\{\x\in\Omega: d(\x)<\delta\}$ --
Lemma~14.6 of \cite{GT}. If $\Omega$ is convex, then the distance
function is superharmonic, i.e., $-\Delta d(\x)$ is a nonnegative
measure. (See Lemma~3 of \cite{AL} for a short proof). For
dimension $n=2$, $-\Delta d\ge 0$ implies that $\Omega$ is convex,
but not for $n>2$. Armitage and Kuran~\cite{AK} give an example of
a torus in dimension greater than 2, which is (obviously) not
convex, but $-\Delta d(\x)\ge 0$.

In order to accommodate weights, we give a small extension of
Theorem~3.1 of Filippas, Maz'ya, and Tertikas \cite{FMT} requiring
only a modification of their change of variable. Rather than
assuming convexity of $\Omega$ it suffices (here and in the proof of
Theorem~3.1 of \cite{FMT}) to assume the weaker condition that
$-\Delta d(\x)\ge 0$ in $\Omega$.
\begin{Theorem}\label{FMT} If $-\Delta d(\x)\ge 0$ in a domain $\Omega$, then for all $u\in H_0^1(\Omega)$,
$\beta<1$, and  $\alpha >\beta-2$
$$
\int_\Omega d(\x)^\beta|\grad u|^2d\x
-\frac{(1-\beta)^2}{4}\int_\Omega \frac{|u|^2}{d(\x)^{2-\beta}}
d\x\ge C_{\alpha,\beta} D_{int}^{\beta-(\alpha+2)}\int_\Omega
d(\x)^\alpha |u|^2d\x
$$
for a constant
$$
C_{\alpha,\beta}:=2^{\alpha -\beta}\cdot\left\{\begin{array}{cc} (\alpha+2-\beta)^2    & \alpha\in (\beta-2,-1)\\
  (1-\beta)(2\alpha +3-\beta)    & \alpha\in [-1,\infty)
  \end{array}\right..
$$
\end{Theorem}
\begin{proof} It will suffice to show the inequality for real-valued $u\in \Con(\Omega)$.
Let $u=d^{\frac{1-\beta}{2}}v$. Since $|\grad d|^2=1$, it follows
that
\begin{equation}\label{Identity}
\int_\Omega d^\beta|\grad u|^2d\x-\frac{(1-\beta)^2}{4}\int_\Omega
d^{\beta-2}u^2d\x =\frac{1-\beta}{2}\int_\Omega (-\Delta d)v^2d\x
+\int_\Omega d|\grad v|^2d\x.
\end{equation}
After noting the identity
$$
\int_\Omega d^\alpha u^2d\x=\int_\Omega d^{\alpha +1-\beta}v^2d\x
$$
we estimate the integral on the right-hand side for $\alpha
>\beta-2$ following a path similar to that of \cite{FMT} to arrive
at their inequality (3.4) and see that for this case
$$\begin{array}{rl}
(\alpha+2-\beta-\delta)\int_\Omega d^{\alpha +1-\beta} v^2d\x \le&
R_{int}^{\alpha+2-\beta}\left (\frac{1}{\delta}\int_{\Omega}d
|\grad v|^2d\x +\int_\Omega (- \Delta d)v^2d\x\right).
\end{array}
$$
Here $R_{int}:=\frac12 D_{int}(\Omega)$. Choose $\delta\le
\min\{\frac{1-\beta}{2}, \frac{\alpha+2-\beta}{2}\}$ and the
conclusion follows.

\end{proof}

If we know that $G$ in Corollary~\ref{CorAppl} is convex, then
$-\Delta d(\x)$ is a positive measure and we may apply
Theorem~\ref{FMT}. 

\begin{Corollary}\label{Corconvex} Assume the hypothesis of Corollary~\ref{CorAppl}. If for $\gamma\in(0,1)$ and $\alpha>\beta -2$
\begin{equation}\label{Qeq}
q_-(\x)\le
(1-\gamma)[\frac{(1-\beta)^2}{4d(\x)^{2-\beta}}+\lambda(G)
d(\x)^\alpha],\q \x\in G,
\end{equation}
for $G=(\Omega\setminus\Omega_{k_0})\cap N_S$ convex and
$\lambda(G)=C_{\alpha,\beta}/D_{int}^{\alpha-(\beta-2)}(G)$, then
$\ft$ is bounded below and closable and the spectrum of $\tilde T$
is purely discrete.

\end{Corollary}
\begin{proof} The proof follows from Corollary~\ref{CorAppl} and Theorem~\ref{FMT}.
\end{proof}

In \cite{FMT2} Filippas, Maz'ya, and Tertikas prove a
Hardy-Sobolev inequality in a tubular domain
$\Omega_\delta:=\{\x\in\Omega: d(\x)<\delta\}$ for some
$\delta>0$. Here, we adapt some of those ideas to use as an
application of Corollary~\ref{CorThm2}. The next Lemma allows
application for the case in which $-\Delta d(\x)\ge 0$ in the
whole of a non-convex $\Omega$, but $G$ in Corollary~\ref{CorAppl}
is not convex and $d$ is not superharmonic in $G$. The prototype
for $\Omega$ in this case is the torus studied by Armitage and
Kuran~\cite{AK}.

It's important to note that in the next Lemma,
$d(\x)=d(\x;\Omega)$, the distance from $\x$ to $\partial\Omega$
as before, as opposed to the distance from $\x$ to $\partial
\Omega_\delta$, $d(\x;\Omega_\delta)$. We will use this additional
notation in some cases below to avoid confusion.

\begin{Lemma}\label{HardyB}
Assume that $\Omega$ is a bounded domain with a $C^2$ boundary and
$-\Delta d\ge 0$ in $\Omega_\delta$ for all $\delta>0$ sufficiently
small. Let $\beta<1$ and $\alpha >(\beta -3)/2$. If $0<\delta\le
\frac{1-\beta}{2}$ then for all $u\in\Con(\Omega)$
$$\begin{array}{rl}
\int_{\Omega_\delta} d^\beta|\grad
u|^2d\x-\frac{(1-\beta)^2}{4}\int_{\Omega_\delta}
d^{\beta-2}|u|^2d\x \ge& C(\alpha,\beta)\delta\int_{\Omega_\delta}
d^{\alpha} |u|^2d\x
\end{array}
$$
for a positive constant
\begin{equation}\label{Calphabeta}
C(\alpha,\beta):=\frac{2^{\alpha-\beta+1}(2\alpha-\beta+3)}{(1-\beta)^{\alpha-\beta+2}}.
\end{equation}
\end{Lemma}
\begin{proof} Since $\Omega$ is bounded and $\partial\Omega\in C^2$, then $d\in C^2(\Omega_{\delta_*}\cap \overline{\Omega})$ for
some $\delta_*$ (Lemma~14.16 of \cite{GT}). We may assume without
loss of generality that $\delta_*=\delta\in
(0,\frac{1-\beta}{2})$. It will suffice to prove the inequality
for functions $u\in \Con(\Omega)$ that are real-valued and
nonnegative (Lieb \& Loss~\cite{LL}, pp.176-177). For
$u\in\Con(\Omega)$ and $u=d^{\frac{1-\beta}{2}}v$ it follows from
integrating by parts that
$$\begin{array}{rl}
\int_{\Omega_\delta} d^\beta|\grad u|^2d\x=&\frac{(1-\beta)^2}{4}
\int_{\Omega_\delta} d^{-1}v^2d\x
+\frac{1-\beta}{2} \int_{\Omega_\delta} (-\Delta d) v^2d\x\\
 &+ \frac{1-\beta}{2}\int_{\partial\Omega^c_\delta}(\grad d\cdot\nu) v^2ds
 +\int_{\Omega_\delta} d|\grad v|^2d\x
\end{array}
$$
since $|\grad d|=1$ where $\nu$ is the unit outward normal from
$\Omega_\delta$
 on $\partial\Omega_\delta \cap \Omega=\partial\Omega^c_\delta$. Since $\grad d\cdot \nu=1$ on $\partial\Omega^c_\delta$
 we have that
\begin{equation}\label{EqAA}\begin{array}{l}
\int_{\Omega_\delta} d^\beta|\grad
u|^2d\x-\frac{(1-\beta)^2}{4}\int_{\Omega_\delta}
d^{\beta-2}u^2d\x
\\
\ \ \ \ \ \ = \frac{1-\beta}{2}\int_{\Omega_\delta} (-\Delta
d)v^2d\x +\int_{\Omega_\delta} d|\grad v|^2d\x +
\frac{1-\beta}{2}\int_{\partial{\Omega}^c_\delta} v^2ds.
\end{array}
\end{equation}

In order to estimate $\int_{\Omega_\delta} d^\alpha u^2 d\x$ for
$\alpha>\frac{\beta-3}{2}>\beta-2$, we make the substitution
$u=d^{\frac{1-\beta}{2}}v$ again and use the identity
$$
div(d^{\alpha-\beta+2}\grad d)=(\alpha-\beta+2)d^{\alpha-\beta+1}
+d^{\alpha-\beta+2}\Delta d
$$
in $\Omega_\delta$. Multiply by $v^2$ and integrate by parts to see
that
\begin{equation}\label{Eqab}\begin{array}{rl}
(\alpha-\beta+2)\int_{\Omega_\delta} d^{\alpha -\beta+1} v^2d\x
=&-2\int_{\Omega_\delta} d^{\alpha-\beta+2}v\grad d\cdot \grad v\
d\x
+\delta^{\alpha-\beta+2}\int_{\partial\Omega_\delta^c} v^2ds\\
&+\int_{\Omega_\delta}d^{\alpha-\beta+2}(-\Delta d) v^2d\x\\
=:& I_1 + I_2 +I_3.
\end{array}
\end{equation}
Applying the Cauchy-Schwarz inequality to $I_1$ we have that
\begin{equation}\label{Eqa31}\begin{array}{rl}
I_1 \le & \delta \int_{\Omega_\delta} d^{\alpha -\beta+1} v^2d\x
+\delta^{\alpha-\beta+1}\int_{\Omega_\delta}d |\grad v|^2d\x
\end{array}
\end{equation}
since $d(\x)\in (0,\delta)$ in $\Omega_\delta$. It follows from
(\ref{Eqab}) and (\ref{Eqa31}) that
\begin{equation}\label{Eqac}\begin{array}{rl}
[(\alpha-\beta+2)-\delta]\int_{\Omega_\delta} d^{\alpha -\beta+1}
v^2d\x
\le & \delta^{\alpha -\beta+1}\int_{\Omega_\delta}d |\grad v|^2d\x\\
&+\delta^{\alpha-\beta+2}\int_{\Omega_\delta}(-\Delta d) v^2d\x\\
&+\delta^{\alpha-\beta+2}\int_{\partial\Omega_\delta^c} v^2ds
\end{array}
\end{equation}
since $-\Delta d(\x)\ge 0$ in $\Omega_\delta$.

Since $\delta\le \frac{1-\beta}{2}$ in (\ref{Eqac}) we will have
that
$$\begin{array}{rl}
\frac{\alpha-\beta+2-\delta}{\delta^{\alpha-\beta+1}}\int_{\Omega_\delta}
d^{\alpha -\beta+1} v^2d\x \le&\frac{1-\beta}{2}\int_{\Omega_\delta}
(-\Delta d)v^2d\x +\int_{\Omega_\delta} d|\grad v|^2d\x
+ \frac{1-\beta}{2}\int_{\partial{\Omega}^c_\delta} v^2ds\\
=&\int_{\Omega_\delta} d^\beta|\grad
u|^2d\x-\frac{(1-\beta)^2}{4}\int_{\Omega_\delta}
d^{\beta-2}u^2d\x
\end{array}
$$
according to (\ref{EqAA}). Finally, use the fact that
$$
C(\alpha,\beta)\delta \le
\frac{\alpha-\beta+2-\delta}{\delta^{\alpha-\beta+1}}
$$
to complete the proof.

\end{proof}

Next, we present a corollary to Lemma~\ref{HardyB} in which we can
use Theorem~\ref{closable} directly avoiding a convexity
assumption for $\Omega\setminus \Omega_{k_0}$. Then, we follow
with an application on a torus in $\R^3$ in which $-\Delta
d(\x)\ge 0$.

\begin{Corollary}\label{torus} Let $\Omega$ be a bounded domain with a $C^2$-boundary and let
$\mathfrak{h}[u,v]$ be given by (\ref{form}) with $\sigma\equiv 0$
and $\mathcal{D}(\mathfrak{h})=\Con(\Omega)$. Set
$S=\partial\Omega$ and define the $S$-admissible family of domains
by
$$
\Omega_k:=\Omega\setminus \overline{\Omega}_{\delta_k},\q
\delta_k=\frac{1}{k}
$$
for $k\in\mathbb{N}$.

 Assume (H)(a),(b),and (c); for some $\beta<1$
$$
\mu_A(\x)\ge d(\x)^\beta,\q \x\in \Omega_{\delta_k};
$$
and $\nu_A\in L^\infty(\Omega\cap\Omega_k)$ for $k$ sufficiently
large. Suppose for $\gamma\in (0,1)$ and $\alpha$ satisfying
$2\alpha -\beta +3>0$
$$
q_-(\x)\le (1-\gamma)[\frac{(1-\beta)^2}{4d(\x)^{2-\beta}}+ \delta
C(\alpha,\beta)\ d(\x)^\alpha],\q \x\in \Omega_{\delta_k}
$$
and $-\Delta d(\x)\ge 0$ in $\Omega_{\delta_k}$ for $k$
sufficiently large and $C(\alpha,\beta)$ defined in
(\ref{Calphabeta}). Then, $\mathfrak{h}$ is bounded below and
closable. The self-adjoint operator associated with $\mathfrak{h}$
has a purely discrete spectrum.
\end{Corollary}
\begin{proof}
It follows from Theorem~\ref{closable} and Lemma~\ref{HardyB} that
$\mathfrak{h}$ is bounded below and closable as well as the fact
that (\ref{lsube}) holds. For $k\ge k_0$, $u\in \mathcal{D}(\ft)$
with $supp\ u\subset \Omega\setminus\Omega_k$ and $\varphi :=
u/\|u\|$
$$\begin{array}{rl}
\ft[\varphi]\ge&\gamma\int_{\Omega_{\delta_k}}<A(\x)\grad \varphi,\grad\varphi>d\x\\
\ge&
\gamma\int_{\Omega_{\delta_k}}[\frac{(1-\beta)^2}{4d(\x)^{2-\beta}}+
\delta C(\alpha,\beta)\ d(\x)^\alpha]|\varphi|^2d\x.
\end{array}
$$
Since $d(\x)\le \frac{1}{k}$ for $\x\in\Omega_{\delta_k}$,
$$\begin{array}{rl}
\ell_e =&\lim_{k\to\infty}\inf_{\varphi}\ft[\varphi]=\infty
\end{array}
$$
implying that the spectrum of the operator associated with the
closure of $\mathfrak{h}$ is discrete.
\end{proof}

\begin{Example}\label{Example} Let $\Omega\subset \R^3$ be the torus obtained by rotating the disc
$\omega =\{(0,y,z): (y-c)^2+z^2<R\}$, $c>2R$, about the $z$-axis.
Armitage and Kuran~\cite{AK} have shown that the distance function
$d_\Omega$ on the whole of $\Omega$ is superharmonic, i.e.,
$-\Delta d_\Omega(\x)\ge 0$ in $\Omega$ although $\Omega$ is not
convex.  Assuming the hypothesis of Corollary~\ref{torus}, the
operator associated with the Dirichlet form $\mathfrak{h}$ on the
torus $\Omega$ has a purely discrete spectrum.
\end{Example}

Of course, the Example~\ref{Example} can be extended to the image of
any unitary transformation of the torus described there since the
spectrum is preserved under such transformations. Note that the
distance function $d_{\Omega_\delta}(\x)$ in $\Omega_\delta$ for
small $\delta>0$
$$
d_{\Omega_\delta}(\x)=\left \{\begin{array}{cc} d_\Omega(\x),& \x\in\Omega_{\delta/2},\\
\delta-d_{\Omega}(\x),& \x\in
\Omega_\delta\setminus\Omega_{\delta/2}
\end{array}\right.
$$
is not superharmonic. Corollary~\ref{Corconvex} does not apply to
the torus of Example~\ref{Example} since $\Omega_\delta$ for
$\delta>0$ is not convex and $d_{\Omega_\delta}$ is not
superharmonic.

Finally, we refer the reader to recent results in \cite{BEL} where Hardy inequalities are given which exploit the interesting
connection between $\Delta d(\x)$ and the principal curvatures at the near point $\y\in\partial\Omega$ of $\x$. These new Hardy inequalities allow for applications of the results here to far more general non-convex domains such as the torus discussed above. Using a representation of $\Delta d$ in terms of principal curvatures, a new proof is given of Armitage and Kuran's result discussed in Example~\ref{Example}.

\bibliographystyle{amsalpha}

\end{document}